\documentclass{amsart}
\usepackage{extarrows}
\usepackage{tikz-cd}
\usetikzlibrary{matrix}

\newtheorem{theorem}{Theorem}[section]
\newtheorem*{theorem*}{Theorem}

 \newtheorem{corollary}[theorem]{Corollary}
 \newtheorem*{corollary*}{Corollary}
 
 \newtheorem{proposition}[theorem]{Proposition}
 \newtheorem*{mainthm*}{Main Theorem}

\theoremstyle{definition}

\theoremstyle{remark}
\newtheorem{remark}[theorem]{Remark}

\begin{document}

\title[Lower bounds for the volume]%
{Lower Bounds for the Volume with Upper Bounds for the Ricci Curvature in Dimension Three}
\author{V.  Gimeno}

\thanks{Research partially supported by the Spanish Government Ministerio de Econom\'ia y Competitividad  (MINECO-FEDER),  grant MTM2017-84851-C2-2-P, and by Universitat Jaume I, grant UJI-B2018-35.}

\maketitle

\begin{abstract}
In this note we provide several lower bounds for the volume of  a geodesic ball within the injectivity radius in a $3$-dimensional Riemannian manifold assuming only upper bounds for the Ricci curvature.
\end{abstract}


\section{Introduction} 
\label{sec:intro}
One of the central topics in Riemannian geometry is the relation between  the curvature of a Riemannian metric defined on a manifold and the behavior of the volume of geodesic balls. Curvature, geodesics and balls have an extremely rich relationship.  A celebrated and well known result  states (see for instance \cite{Chavel2}) that if a $n$-dimensional Riemannian manifold $(M,g)$ has the sectional curvatures ${\rm sec}_M(\Pi)$  of any tangent plane $\Pi$  bounded from above by a constant $\kappa$,
$$
{\rm sec}_M(\Pi)\leq \kappa 
$$
then, for any point $p\in M$, the volume ${\rm V}(p,t)$ of the geodesic ball of radius $t$ centered at $p$ is bounded from below by
\begin{equation}\label{Bishop-Gunter}
\displaystyle{\rm V}_M(p,t)\geq {\rm V}_{\mathbb{M}_\kappa^n}(t)
\end{equation}
for any $t\leq \min\{{\rm inj}(p),\pi/\sqrt{\kappa}\}$\footnote{In this paper is understood $\pi/\sqrt{\kappa}=+\infty$ when $\kappa\leq 0$}, where  ${\rm V}_{\mathbb{M}_\kappa^n}(t)$ is the volume of the geodesic ball of radius $t$ in the simply-connected real space form $\mathbb{M}_\kappa^n$ of dimension $n$ and constant sectional curvature $\kappa$.  This inequality was obtained by Bishop and G\"unter and it has associated a rigidity result: if equality is attained in inequality (\ref{Bishop-Gunter}), the geodesic ball of radius $t$ in $M$ centered at $p\in M$ is isometric to the geodesic ball of radius $t$ in $\mathbb{M}_\kappa^n$.
 
An other classical result authored by Bishop and Gromov,  (see \cite{Chavel2}) states that whenever the Ricci curvatures are bounded from below by
$$
{\rm Ric}\geq (n-1)\kappa,
$$
the volume of the geodesic ball of radius $t$ is bounded by from above by
$$
\displaystyle{\rm V}_M(p,t)\leq {\rm V}_{\mathbb{M}_\kappa^n}(t)
$$
for any $t>0$.

Furthermore, Calabi and Yau (see \cite{Yau1976}) proved  that for any  complete and non-compact Riemannian manifold with 
$$
{\rm Ric}\geq 0
$$ 
there exists a constant $C$ such that the volume of the geodesic ball is bounded from below by
$$
{\rm V}_M(p,t)\geq C\, t.
$$
We would like to stress here that, in the above theorems, upper bounds are imposed only on the sectional curvature, and for the Ricci curvature only lower bounds are  used. The goal of this paper is to obtain lower bounds for the volume of geodesics balls when the Ricci curvature is bounded from above. This objective is achieved in dimension $3$. The results of this paper are detailed in the following section.

\section{Main Results}\label{mainsec}

Our first result is a Bishop-G\"unter type inequality but using bounds on the Ricci curvature:
\begin{theorem}\label{teo-unua}Let $(M,g)$ be a $3$-dimensional Riemannian manifold. Suppose that
$$
{\rm Ric}\leq 2\kappa.
$$Then, for any $p\in M$ and for any $t\leq \min\{{\rm inj}(p),\pi/\sqrt{\kappa}\}$,
the volume ${\rm V}_M(p,t)$ of the geodesic ball of radius $t$ centered at $p$ is bounded from below by
\begin{equation}
\displaystyle{\rm V}_M(p,t)\geq {\rm V}_{\mathbb{M}_\kappa^3}(t),
\end{equation}
where  ${\rm V}_{\mathbb{M}_\kappa^3}(t)$ is the volume of the geodesic ball of radius $t$ in the simply-connected real space form $\mathbb{M}_\kappa^3$ of dimension $3$ and constant sectional curvature $\kappa$.
\end{theorem}
The hypothesis of the above Theorem implies global upper bounds in the  Ricci curvature. In the following Theorem we are  assuming that the positive upper bound of the Ricci curvature has finite $L_1$-norm in $M$. More precisely, for any point $q\in M$ let us denote by $K_+:M\to \mathbb{R}$ the function 
$$
q\mapsto K_+(q)=\max\left\lbrace 0,\{\max\left\lbrace{\rm Ric}(v,v)\, \,  :\, \, v\in T_qM\, {\rm with} \, \Vert v\Vert =1\right\rbrace\ \right\rbrace
$$
Under the hypothesis of finite $L_1$-norm of this $K_+$ function we obtain the following Theorem
\begin{theorem}\label{teo-dua}
Let $(M,g)$ be a $3$-dimensional Riemannian manifold. Suppose that
$$
\int_MK_+ d{\rm V}_g=C <\infty.
$$Then, for any $p\in M$ and for any $t\leq {\rm inj}(p)$,
the volume ${\rm V}_M(p,t)$ of the geodesic ball of radius $t$ centered at $p$ is bounded from below by
$$
{\rm V}_M(p,t)\geq \frac{4}{3}\pi t^3-Ct^2.
$$
\end{theorem}
An immediate consequence of the above Theorem is the following Corollary when we restrict ourselves to manifolds with pole 

\begin{corollary}\label{cor-infinite-volume}
Let $(M,g)$ be a $3$-dimensional Riemannian manifold with a point $p\in M$ with empty cut locus ${\rm Cut}(p)=\emptyset$, suppose that $$
\sup_{M}{\rm Ric}<\infty.
$$
Then $M$ has infinite total volume, \emph{i.e.},
$$
{\rm vol}(M)=\infty.
$$
\end{corollary}

\begin{remark}
Observe that in the above Corollary, and in the main results of this paper, the assumption that $M$ contains a point with empty cut locus can not be removed. Indeed, Lohkamp proved in \cite{Lohkamp1994} that each manifold $M^n$, $n\geq 3$ admits a complete metric with  negative Ricci curvature and finite volume, ${\rm vol}(M)<\infty$. 
\end{remark}
\begin{remark}
The Hypothesis of dimension $3$ is used in the proof of Theorem \ref{teo-unua} (see Section \S \ref{ProofSec}) because in such a case the  integral of the (intrinsic) scalar curvature of the $2$-dimensional geodesic spheres is a topological invariant via the Gauss-Bonnet Theorem. We would like to remark here that the assumption that $M$ has dimension $3$ can not be removed. For $n>3$, upper bounds on the Ricci curvature are not enough to state the result. See for instance \cite{Kloeckner2015} where  the ``$\sqrt{{\rm Ric}}$'' curvature is used to obtain Günter type inequalities for any dimension. The Bishop-Gromov Theorem can not be reversed: the statement that if ${\rm Ric}\leq (n-1) \kappa$ then ${\rm V}(p,t)\geq {\rm V}_{\mathbb{M}_\kappa^n}(t)$ is not true for $n>3$.  Consider, for example, the complex hyperbolic space $(\mathbb{C}H_b^{N},g_b)$ of dimension $2N$ endowed with the metric $g_b$ of constant holomorphic sectional curvature $-4b$. The Ricci tensor and the volume of a geodesic ball of radius $t$ are given by (see \cite{Tubes} for instance)
\begin{equation}\label{Einstein}
{\rm Ric}=-2b(N+1) g_b,\quad {\rm V}(t)=\frac{1}{N!}\left(\frac{\pi}{b}\right)^N\left(\sinh(t\sqrt{b})\right)^{2N}.
\end{equation}
Then there exists a constant $B(N)$ such that 
$$
{\rm V}(t)\leq B(N) e^{2N \sqrt{b}t}.
$$
But observe that (\ref{Einstein}) is compatible with the bound
$$
{\rm Ric}\leq (n-1)\kappa,\quad\kappa=\frac{-2b(N+1)}{2N-1},
$$
and
$$
{\rm V}_{\mathbb{M}_\kappa^{2N}}(t)=\frac{C_n}{\sqrt{-\kappa}}\int_0^t\sinh^{2N-1}(\sqrt{-\kappa}s)ds.
$$
Hence, there exists an other constant $C(N)$ such that 
$$
{\rm V}_{\mathbb{M}_\kappa^{2N}}(t)\geq C(N) e^{\sqrt{2(N+1)}\sqrt{2N-1}\sqrt{b}t},\quad {\rm for}\quad t\geq 1.
$$
Since $\sqrt{2(N+1)}\sqrt{2N-1}>2N$ for $N>1$ (dimension $2N>3$), then there exists $t$ large enough such that
$$
{\rm V}(t)< {\rm V}_{\mathbb{M}_\kappa^{2N}}(t).
$$
\end{remark}
\begin{remark}
In some cases the use of upper bounds for the Ricci curvature is more appropriate than the use of upper bounds for the sectional curvature. There are several examples where the bounds obtained in Theorem \ref{teo-unua} are better than the bounds obtained when the classical Bishop-G\"unter inequality is used. Every Berger sphere is one of these examples. 
Let $SU(2)$ be the special unitary group of $2\times 2$ matrices,
$$
\begin{aligned}
SU(2):=&\left\{A \in M_{2\times 2}(\mathbb{C})\, :\, {\rm det}(A)=1,\, A^\dag =A^{-1}\right\}\\
=&\left\{\begin{pmatrix}
z_1&z_2\\
-\overline{z_2}& \overline{z_1}
\end{pmatrix}\, :\, \vert z_1\vert^2+\vert z_2\vert^2=1 \right\}\\
=& \mathbb{S}^3(1).
\end{aligned}
$$
The Lie algebra $\mathfrak{su}(2)$ is given by
$$
\mathfrak{su}(2)={\rm span}_\mathbb{R}\{X_1,X_2,X_3\}
$$
with 
$$
X_1:=\begin{pmatrix}
i&0\\
0&-i
\end{pmatrix},\quad X_2:=\begin{pmatrix}
0&1\\
-1&0
\end{pmatrix},\quad X_3:=\begin{pmatrix}
0&i\\
i&0
\end{pmatrix}\cdot
$$
For $0<\epsilon<1$, Let $g_\epsilon$ be the metric such that
$$
\left\{\frac{X_1}{\epsilon},X_2,X_3\right\}
$$
is an orthonormal basis, then (see \cite{Petersen}),
the sectional curvatures of any tangent plane $\pi$ are bounded therefore by 
$
{\rm sec}_M(\pi)\leq 4-3\epsilon^2
$
and
$
{\rm Ric}\leq 4-2\epsilon^2
$.

Since $0<\epsilon<1$,  we have 
$
2-\epsilon^2\leq 4-3\epsilon^2 
$
and the bound given by Theorem \ref{teo-unua},
\begin{equation}
{\rm V}_M(p,t)\geq {\rm V}_{\mathbb{M}_{2-\epsilon^2}^3}(t),\quad {\rm for}\quad  t\leq\min\left\{{\rm inj}(p), 	\frac{\pi}{\sqrt{2-\epsilon^2}}\right\},
\end{equation} 
is better than the bound given by the classical inequality obtained by using the Bishop-G\"unter inequality, \emph{i.e.},
\begin{equation}
{\rm V}_M(p,t)\geq {\rm V}_{\mathbb{M}_{4-3\epsilon^2}^3}(t),\quad {\rm for}\quad  t\leq\min\left\{{\rm inj}(p), 	\frac{\pi}{\sqrt{4-3\epsilon^2}}\right\}.
\end{equation} 
\end{remark}
\begin{remark}Croke in \cite{Croke1980} proved that for any complete $n$-dimensional manifold  $(M,g)$ the volume of the geodesic ball of radius $t$ centered at $p\in M$ is bounded from below by
$$
{\rm V}(p,t)\geq \frac{2^{n-1}\omega_{n-1}^n}{\omega_n^{n-1}n^n} t^n,\quad {\rm for}\quad t\leq \frac{1}{2}{\rm inj}(M)
$$
with $\omega_n$ the volume of the unit $n$-sphere in $\mathbb{R}^{n+1}$. In the particular case of dimension $3$, 
$$
{\rm V}(p,t)\geq \left(\frac{4}{3}\right)^3\frac{1}{\pi} t^3,\quad {\rm for}\quad t\leq \frac{1}{2}{\rm inj}(M).
$$
By using Theorem \ref{teo-unua}, if ${\rm Ric}\leq 2 \kappa$ in the geodesic ball of radius $t$ centered at $p$, we can provide the following improvement:
$$
{\rm V}(p,t)\geq {\rm V}_{\mathbb{M}_\kappa^3}(t)>\left(\frac{4}{3}\right)^3\frac{1}{\pi} t^3,\quad {\rm for}\quad t\leq \min\{{\rm inj}(p),\pi/\sqrt{\kappa}\}.
$$
In \cite{Berger1980} Berger proved that for a compact $n$-dimensional manifold  $(M,g)$,
$$
{\rm vol}(M)\geq \frac{\omega_n}{\pi^n}{\rm inj}(M)^n. 
$$
When $n=3$, this equation can be rewritten as
\begin{equation}\label{volinj}
{\rm vol}(M)\geq \frac{2}{\pi}{\rm inj}(M)^3. 
\end{equation}
Since $M$ is compact there exists a constant $\kappa$ such that ${\rm Ric}(v,v)\leq \kappa g(v,v)$ for any $v\in T_pM$ and any $p\in M$. Then by using Theorem \ref{teo-unua},
$$
{\rm vol}(M)\geq {\rm V}_{\mathbb{M}_\kappa^3}(R),\quad R=\min\{{\rm inj}(M),\pi/\sqrt{\kappa}\}
$$
hence, when $\kappa\leq 0$ we can provide the following improvement of (\ref{volinj})
$$
{\rm vol}(M)\geq \frac{4\pi}{3}{\rm inj}(M)^3 .
$$
\end{remark}

\section{Proof of the Main Results}\label{ProofSec}The proof of the main results follows from Proposition \ref{Gauss-Bonnet} which is an adaptation of the area variation formula for geodesic balls in dimension $3$ taking into account that the total integral of the scalar curvature in a geodesic sphere is a topological invariant (for $2$-dimensional spheres).

Let $(M,g)$ be a $n$-dimensional Riemannian manifold, let $p\in M$ be a point of $M$, and let ${\rm inj}(p)$ denote the injectivity radius of $p$. Let ${\rm B}_{{\rm inj}(p)}(0)$ be the ball of radius ${\rm inj}(p)$ centered at $0$ in $T_pM$, let ${\rm  B}_{{\rm inj}(p)}(p)=\exp_p({\rm B}_{{\rm inj}(p)}(0))$ the geodesic ball of radius ${\rm inj}(p)$ centered at $p$, then the exponential map
$$
\exp_p: {\rm B}_{{\rm inj}(p)}(0)\longrightarrow {\rm  B}_{{\rm inj}(p)}(p),
$$
is a diffeomorphism.  The radial vector field $\partial r$ is globally defined on ${\rm  B}_{{\rm inj}(p)}(p)\setminus\{p\}$ and is given by
$$
\partial r:{\rm  B}_{{\rm inj}(p)}(p)\setminus\{p\}\to TM\setminus\{p\},\quad q\mapsto\partial r (q)=\frac{d}{dt}\left.\exp_p\left(r(q)\theta(q)t\right)\right\vert_{t=1}.
$$ 
Namely, $\partial r(q)$ is the tangent vector to the arc-length parametrized geodesic curve from $p$ to $q$.

Let us denote by $d{\rm V}_g$ the Riemannian volume form associated to $g$. The volume ${\rm V}_M(p,t)$ of the geodesic ball $B_t(p)$ of radius $t$  centered at $p$ is given by
$$
{\rm V}_M(p,t)=\int_{B_t(p)}d{\rm V}_g
$$   
The vector field  $\partial r$ coincides with the gradient $\nabla r$  of the polar radius function $r$ on ${\rm  B}_{{\rm inj}(p)}(p)\setminus \{p\}$, \emph{i.e.}, $\partial r=\nabla r$, and furthermore $\Vert \nabla r\Vert=\Vert \partial r\Vert=1$. Moreover since the geodesic sphere $S_t(p)$ of radius $t$ centered at $p$ is a level set of $r$, \emph{i.e.}, $S_t(p)=r^{-1}(t)$ then the vector field $\nabla r$ is a unit vector field normal and pointed outward to $S_t(p)$. 

The volume ${\rm A}_M(p,t)$ of the  geodesic sphere $S_t(p)$ of radius $t$ centered at $p$, is given therefore by
$$
{\rm A}_M(p,t)=\int_{S_t(p)}\nabla r\lrcorner d{\rm V}_g
$$  
where $\nabla r\lrcorner d{\rm V}_g$ is the contraction of the Riemannian volume form $d{\rm V}_g$ with the vector field $\nabla r$. In order to simplify the notation we will make use of
$$
d{\rm A}_g:=\nabla r\lrcorner d{\rm V}_g
$$
Note that for $0<t < {\rm inj}(p)$,  the function 
$$
t\mapsto {\rm V}_M(p,t)
$$
is smooth and with derivative ${\rm A}_M(p,t)$. The second fundamental form $\alpha$ of the inclusion map from $S_t(p)$ to $M$ is given in terms of the Hessian ${\rm Hess}_Mr$ of the geodesic distance function $r$ to the pole $p$ because for any two vector fields $X,Y \in \mathfrak{X}(S_t(p))$
$$
\begin{array}{lcl}
\alpha(X,Y)&=&\langle \nabla_XY,\nabla r\rangle\nabla r=\left(X(\langle  Y,\nabla r\rangle)-\langle Y,\nabla_X\nabla r\rangle\right)\nabla r\\
&=&-\langle Y,\nabla_X\nabla r\rangle\nabla r
=-{\rm Hess}_Mr(X,Y)\nabla r.
\end{array}
$$   
The mean curvature vector field $\vec H $ of $S_t(p)$ is given therefore in terms of the Laplacian $\Delta_Mr$ of the distance function to $p$ because for any $q\in S_t(p)$ and any orthonorlmal basis $\{E_i\}_{i=1}^{n-1}$ of $T_q S_t(p)$ 
$$
\vec{H}=\sum_{i=1}^{n-1} \alpha(E_i,E_i)=-\sum_{i=1}^{n-1} {\rm Hess}_Mr(E_i,E_i)\nabla r=-\Delta_M r\nabla r
$$
The following Proposition states the first and second variation formula for the area function $t\mapsto {\rm A}_M(t)$,
\begin{proposition}[See pages 4 and 8 of  \cite{LiBook}]\label{prop1}
Let $(M,g)$ be a Riemannian manifold, suppose that $p\in M$ and $t<{\rm inj}(p)$. Then,
\begin{enumerate}
\item The first derivative ${\rm A}_M'(p,t)$ with respect to $t$ of the volume ${\rm A}_M(p,t)$ of the geodesic   sphere $S_t(p)$ of radius $t$ centered at $p$ is given by
$$
{\rm A}_M'(p,t)=\int_{S_t(p)}\Delta_Mrd{\rm A}_g=\int_{S_t(p)}Hd{\rm A}_g.
$$
\item The second derivative ${\rm A}_M''(p,t)$ with respect to $t$ of the volume ${\rm A}_M(p,t)$ of the geodesic  sphere $S_t(p)$  of radius $t$ centered at $p$ is given by
$$
\begin{array}{lcl}
\displaystyle{\rm A}_M''(p,t)&=&\displaystyle\int_{S_{t}(p)}\left(-{\rm Ric}(\nabla r,\nabla r)-\Vert {\rm Hess}_Mr\Vert^2+\left(\Delta_Mr\right)^2\right)d{\rm A}_g
\end{array}
$$ 
\end{enumerate}  
\end{proposition}

Let us choose now an orthonormal basis $\{E_1,\cdots,E_{n-1},\nabla r\}$ of $T_qM$ which diagonalizes ${\rm Hess}_Mr$, \emph{i.e.},
$$
{\rm Hess}_Mr(E_i,E_j)=\left\lbrace\begin{array}{lcl}
\lambda_i&{\rm if }& i=j\\
0 &{\rm if }& i\neq j
\end{array}\right.
$$
Then, 
$$
\begin{array}{lcl}
 -\Vert {\rm Hess}_Mr\Vert^2+\left(\Delta_Mr\right)^2&=&\displaystyle -\sum_{i=1}^{n-1}\lambda_i^2+\left(\sum_{i=1}^{n-1}\lambda_i\right)^2\\
 &=&\displaystyle -\sum_{i=1}^{n-1}\lambda_i^2+\left(\sum_{i=1}^{n-1}\lambda_i\right)\left(\sum_{j=1}^{n-1}\lambda_j\right)\\
 & =&\displaystyle -\sum_{i=1}^{n-1}\lambda_i^2+\sum_{i,j=1}^{n-1}\lambda_i\lambda_j=\sum_{i\neq j}^{n-1}\lambda_i\lambda_j
\end{array}
$$
Taking into account that $\alpha(E_i,E_j)=-{\rm Hess}_Mr(E_i,E_j)\nabla r$ and using the Gauss formula (see \cite{Oneill} for instance) 
$$
-\Vert {\rm Hess}_Mr\Vert^2+\left(\Delta_Mr\right)^2=\sum_{i\neq j}^{n-1}\left({\rm sec}_{S_t(p)}(E_i,E_j)-{\rm sec}_{M}(E_i,E_j)\right)
$$
Finally, it is easy to check that
$$
-{\rm Ric}(\nabla r,\nabla r)-\Vert {\rm Hess}_Mr\Vert^2+\left(\Delta_Mr\right)^2={\rm scal}_{S_t(p)}+{\rm Ric}(\nabla r,\nabla r)-{\rm scal}_M,
$$
where ${\rm scal}_M$ is the scalar curvature function of $M$, ${\rm Ric}(\nabla r, \nabla r)$ is the Ricci tensor evaluated in $\nabla r$, and ${\rm scal}_{S_t(p)}$ is the intrinsic scalar curvature of the sphere $S_t(p)$. 
Then
\begin{equation}
\begin{array}{lcl}
\displaystyle{\rm A}_M''(p,t)
&=&\displaystyle\int_{S_t(p)}\left({\rm scal}_{S_t(p)}+{\rm Ric}(\nabla r,\nabla r)-{\rm scal}_M\right)d{\rm A}_g
\end{array}
\end{equation}

When the dimension of $M$ is $3$, the geodesic sphere $S_t(p)$ of radius $t$ centered at $p$ has dimension $2$ and the scalar curvature is given in terms of the Gaussian curvature $K_G$,
$$
{\rm scal}_{S_t(p)}=2 K_G
$$
and by the Gauss-bonnet Theorem
$$
\int_{S_t(p)}{\rm scal}_{S_t(p)}d{\rm A}_g=2\int_{S_t(p)}K_G d{\rm A}_g =4\pi\chi(S_t(p))=8\pi
$$
therefore we can state the following Corollary to Proposition \ref{prop1}
\begin{corollary}\label{cordim3}
Let $(M,g)$ be a $3$-dimensional Riemannian manifold, let $p\in M$ be a point of $M$. Then for any $0<t<{\rm inj}(p)$,
the second derivative ${\rm A}_M''(p,t)$ with respect to $t$ of the volume ${\rm A}_M(p,t)$ of the geodesic  sphere $S_t(p)$  of radius $t$ centered at $p$ is given by
$$
\begin{array}{lcl}
\displaystyle{\rm A}_M''(p,t)
&=&\displaystyle8\pi-\int_{S_t(p)}\left({\rm scal}_M-{\rm Ric}(\nabla r,\nabla r)\right)d{\rm A}_g.
\end{array}
$$ 
where ${\rm scal}_M$ is the scalar curvature function of $M$, and ${\rm Ric}(\nabla r, \nabla r)$ is the Ricci tensor evaluated in $\nabla r$.  
\end{corollary}
Since ${\rm V}_M'(p,t)={\rm A}_M(p,t)$,
$$
{\rm V}_M''(p,s)\vert_{s=t}-{\rm V}_M''(p,t_0)=\int_{t_0}^t{\rm A}_M''(p,s)ds
$$
Therefore using the above Corollary,
\begin{equation}\label{eqkvara}
\begin{aligned}
\displaystyle{\rm V}_M''(p,t)-\int_{S_{t_0}(p)}Hd{\rm A}_g=&8\pi(t-t_0)\\&-\int_{t_0}^t\int_{S_s(p)}\left({\rm scal}_M-{\rm Ric}(\nabla r,\nabla r)\right)d{\rm A}_gds
\end{aligned}
\end{equation}

Taking the limit $t_0\to 0$ we obtain the following

\begin{proposition}\label{Gauss-Bonnet}Let $(M,g)$ be a $3$-dimensional Riemannian manifold, let $p\in M$ be a point of $M$.  Then for any $0<t<{\rm inj}(p)$
$$
\displaystyle\int_{B_t(p)}\left({\rm scal}_M-{\rm Ric}(\nabla r,\nabla r)\right)d{\rm V}_g+{\rm V}_M''(p,t)=8\pi t
$$
\end{proposition}
\begin{proof}
The Proposition follows taking the limit $t_0\to 0$ in equation (\ref{eqkvara}) because in dimension $3$
$$
\displaystyle \lim_{t_0\to 0}\int_{S_{t_0}(p)}Hd{\rm A}_g=\lim_{t_0\to 0}\int_{S_{t_0}(p)}\Delta_Mr d{\rm A}_g=0.
$$
Indeed, in \cite{Chen1981} for example, it is proved that
$$
H(p)=\frac{n-1}{t}+O(t),\quad {\rm A}_M(p,t)=C_{n-1}t^{n-1}+O(t^{n+1}),\quad {\rm as}\quad t\to 0.
$$\end{proof}
\begin{corollary}\label{constant-curvature}
Let $\mathbb{M}_\kappa^3$ be the $3$-dimensional simply-connected real space form  of constant sectional curvature $\kappa$, then
$$
4\kappa{\rm V}_{\mathbb{M}_\kappa^3}(t)+{\rm V}_{\mathbb{M}_\kappa^3}''(t)=8\pi t
$$
\end{corollary}
From Proposition \ref{Gauss-Bonnet} we can prove the main results of the paper

\subsection{Proof of Theorem \ref{teo-unua} }The statement and proof of Theorem \ref{teo-unua} is as follows
\begin{theorem*}Let $(M,g)$ be a $3$-dimensional Riemannian manifold. Suppose that
$$
{\rm Ric}\leq 2\kappa.
$$Then, for any $p\in M$ and for any $t\leq \min\{{\rm inj}(p),\pi/\sqrt{\kappa}\}$,
the volume ${\rm V}_M(p,t)$ of the geodesic ball of radius $t$ centered at $p$ is bounded from below by
\begin{equation}
\displaystyle{\rm V}_M(p,t)\geq {\rm V}_{\mathbb{M}_\kappa^3}(t),
\end{equation}
where  ${\rm V}_{\mathbb{M}_\kappa^3}(t)$ is the volume of the geodesic ball of radius $t$ in the simply-connected real space form $\mathbb{M}_\kappa^3$ of dimension $3$ and constant sectional curvature $\kappa$.\end{theorem*}
\begin{proof}
Let $\{\nabla r,E_1,E_2\}$ be an orthonormal basis of $T_qM$. Since ${\rm Ric}\leq 2\kappa $,
$$
{\rm scal}_M-{\rm Ric}(\nabla r,\nabla r)={\rm Ric}(E_1,E_1)+{\rm Ric}(E_2,E_2)\leq 4\kappa
$$
By using Proposition \ref{Gauss-Bonnet} and Corollary \ref{constant-curvature}
\begin{equation}\label{eqKvina}
\begin{aligned}
4\kappa {\rm V}_M(p,t)+{\rm V}_M''(p,t)\geq & \displaystyle\int_{B_t(p)}\left({\rm scal}_M-{\rm Ric}(\nabla r,\nabla r)\right)d{\rm V}_g+{\rm V}_M''(p,t)\\
=&8\pi t =4\kappa{\rm V}_{\mathbb{M}_\kappa^3}(t)+{\rm V}_{\mathbb{M}_\kappa^3}''(t)
\end{aligned}
\end{equation}
Let us denote by $Z(t):= {\rm V}_M(p,t)-{\rm V}_{\mathbb{M}_\kappa^3}(t)$,  and by
$$
{\rm sn}_\kappa(t):=\left\{\begin{array}{lcr}
\sinh(\sqrt{-\kappa}t)&{\rm if}&\kappa<0\\
t&{\rm if}&\kappa=0\\
\sin(\sqrt{\kappa}t)&{\rm if}&\kappa>0
\end{array}\right.
$$
inequality (\ref{eqKvina}) can be rewritten as
$$ 
Z''(t)\geq -4\kappa Z(t)=\frac{{\rm sn}_{4\kappa}
''(t)}{{\rm sn}_{4\kappa}(t)}Z(t)
$$
which implies 
$$
\frac{d}{dt}\left( Z'(t){\rm sn}_{4k}(t)-Z(t){\rm sn}'_{4k}(t)\right)\geq 0
$$
Since $Z'(t){\rm sn}_{4k}(t)-Z(t){\rm sn}'_{4k}(t)$ is a non-decreasing function
$$
Z'(t){\rm sn}_{4k}(t)-Z(t){\rm sn}'_{4k}(t)\geq \lim_{t\to 0}\left(Z'(t){\rm sn}_{4k}(t)-Z(t){\rm sn}'_{4k}(t)\right)=0
$$
Therefore
$$
\frac{d}{dt}\left(\frac{Z(t)}{{\rm sn}_{4k}(t)}\right)\geq 0
$$
then, taking into account that ${\rm V}_{\mathbb{M}_\kappa^3}(t)\sim Ct^3+O(t^4)$, ${\rm V}_{M}(p,t)\sim Ct^3+O(t^4)$, and ${\rm sn}_{4k}(t)\sim t+O(t^3)$ when $t$ tend to zero, 
$$
\frac{Z(t)}{{\rm sn}_{4k}(t)}\geq \lim_{t\to 0}\left(\frac{Z(t)}{{\rm sn}_{4k}(t)}\right)=\lim_{t\to 0}\left(\frac{{\rm V}_M(p,t)}{{\rm sn}_{4k}(t)}-\frac{{\rm V}_{\mathbb{M}_\kappa^3}(t)}{{\rm sn}_{4k}(t)}\right)=0.
$$
Therefore 
$$
\begin{array}{lcr}
Z(t)\geq 0  &\Longrightarrow & {\rm V}_{M}(p,t)\geq {\rm V}_{\mathbb{M}_\kappa^3}(t).
\end{array}
$$\end{proof}
\subsection{Proof of Theorem \ref{teo-dua}}
The statement and proof of Theorem \ref{teo-dua} is as follows
\begin{theorem*}Let $(M,g)$ be a $3$-dimensional Riemannian manifold. Suppose that
$$
\int_MK_+ d{\rm V}_g=C <\infty.
$$Then, for any $p\in M$ and for any $t\leq {\rm inj}(p)$,
the volume ${\rm V}_M(p,t)$ of the geodesic ball of radius $t$ centered at $p$ is bounded from below by
$$
{\rm V}_M(p,t)\geq \frac{4}{3}\pi t^3-Ct^2.
$$
\end{theorem*}
\begin{proof}
Let $\{\nabla r,E_1,E_2\}$ be an orthonormal basis of $T_qM$. Since ${\rm Ric}\leq K_+ $,
$$
{\rm scal}_M-{\rm Ric}(\nabla r,\nabla r)={\rm Ric}(E_1,E_1)+{\rm Ric}(E_2,E_2)\leq 2K_+(q)
$$
By using Proposition \ref{Gauss-Bonnet}

$$
\begin{aligned}
V_M''(p,t)=&8\pi t-\displaystyle\int_{B_t(p)}\left({\rm scal}_M-{\rm Ric}(\nabla r,\nabla r)\right)d{\rm V}_g\\
\geq & 8\pi t-2\displaystyle\int_{M}K_+ d{\rm V}_g=8\pi t-2C
\end{aligned}
$$
and the Theorem follows integrating twice and taking into account that ${\rm V}_M(p,0)={\rm V}_M'(p,0)=0$,
$$
{\rm V}_M(p,t)\geq \frac{4}{3}\pi t^3-Ct^2.
$$
\end{proof}

\def\cprime{$'$} \def\polhk#1{\setbox0=\hbox{#1}{\ooalign{\hidewidth
  \lower1.5ex\hbox{`}\hidewidth\crcr\unhbox0}}}
  \def\polhk#1{\setbox0=\hbox{#1}{\ooalign{\hidewidth
  \lower1.5ex\hbox{`}\hidewidth\crcr\unhbox0}}}
  \def\polhk#1{\setbox0=\hbox{#1}{\ooalign{\hidewidth
  \lower1.5ex\hbox{`}\hidewidth\crcr\unhbox0}}} \def\cprime{$'$}
  \def\cprime{$'$} \def\cprime{$'$} \def\cprime{$'$} \def\cprime{$'$}

\end{document}